\documentclass[a4paper,10pt,reqno]{amsart}

\usepackage[UKenglish]{babel}

\usepackage{amsmath}
\usepackage{amssymb}
\usepackage{amsthm,bbm,bm}
\usepackage{verbatim}
\usepackage{stackrel}

\usepackage{comment,cite}	
\usepackage{hyperref}
\hypersetup{
    colorlinks=true,
    linkcolor=blue,
    citecolor=red,
    filecolor=magenta,      
    urlcolor=cyan,
    pdftitle={Order on extrapolation space},
    linktocpage=true,
}

\usepackage[utf8]{inputenc}

\usepackage[shortlabels]{enumitem}

\usepackage{marginnote}
\usepackage{color}

\usepackage{stmaryrd}
\usepackage{tikz}
\usepackage{tikz-cd}

\newcommand{\bbR}{\mathbb{R}}

\newcommand{\calL}{\mathcal{L}}

\DeclareMathOperator{\id}{id} 
\DeclareMathOperator{\dist}{dist} 
\newcommand{\argument}{\mathord{\,\cdot\,}} 
\newcommand{\dxShort}{\mathrm{d}} 
\DeclareMathOperator{\linSpan}{span} 
\newcommand{\norm}[1]{\left\lVert #1 \right\rVert} 
\newcommand{\modulus}[1]{\left\lvert #1 \right\rvert} 
\newcommand{\duality}[2]{\left\langle#1\, ,\, #2\right\rangle} 
\newcommand{\dom}[1]{\operatorname{dom}\left(#1\right)} 

\newcommand{\openBall}[2]{B_{<#1}(#2)}

\newcommand{\blNorm}[1]{\norm{#1}_{\mathrm{BL}}} 

\theoremstyle{plain}

\theoremstyle{definition}
\newtheorem{definition}{Definition}[section]
\newtheorem{remark}[definition]{Remark}

\newtheorem*{remark*}{Remark}
\newtheorem*{remarks*}{Remarks}
\newtheorem{example}[definition]{Example}

\newtheorem{proposition}[definition]{Proposition}
\newtheorem{lemma}[definition]{Lemma}
\newtheorem{theorem}[definition]{Theorem}
\newtheorem{corollary}[definition]{Corollary}

\newtheorem{open_problem}[definition]{Open Problem}
\newtheorem{open_problems}[definition]{Open Problems}

\numberwithin{equation}{section} 

\begin{document}

\title[Negative Sobolev and extrapolation spaces]{The lattice structure of negative Sobolev and extrapolation spaces}

\author{Sahiba Arora}
\address{Department of Applied Mathematics, University of Twente, 217, 7500 AE, Enschede, The Netherlands}
\email{sahiba.arora@math.uni-hannover.de}

\author{Jochen Glück}
\address{University of Wuppertal, School of Mathematics and Natural Sciences, Gaußstr.\ 20, 42119 Wuppertal, Germany}
\email{glueck@uni-wuppertal.de}

\author{Felix L. Schwenninger}
\address{Department of Applied Mathematics, University of Twente, 217, 7500 AE, Enschede, The Netherlands}
\email{f.l.schwenninger@utwente.nl}

\begin{abstract}
    It is well-known that the Sobolev spaces $W^{k,p}(\mathbb{R}^d)$ are vector lattices with respect to the pointwise almost everywhere order if $k \in \{0,1\}$, but not if $k \ge 2$. 
    In this note, we consider negative indices $-k$ and show that the span of the positive cone in $W^{-k,p}(\mathbb{R}^d)$ is a vector lattice in this case. 
    On bounded domains $\Omega \subseteq \mathbb{R}^d$ we  obtain a partial result in this direction. 
    
    We also prove a related abstract theorem:~if $(T(t))_{t \in [0,\infty)}$ is a positive $C_0$-semigroup on a Banach lattice $X$ with order continuous norm, then the span of the cone $X_{-1,+}$ in the extrapolation space $X_{-1}$ is a vector lattice. 
    This complements results obtained by Bátkai, Jacob, Wintermayr, and Voigt in the context of perturbation theory and provides additional context for the theory of infinite-dimensional positive systems.
\end{abstract}


\date{\today}

\maketitle

\section{Introduction}

For every $p \in [1,\infty]$, the first order Sobolev space $W^{1,p}(\bbR^d)$ is a vector lattice with respect to the pointwise almost everywhere order. 
This follows from the fact that $\modulus{f} \in W^{1,p}(\bbR^d)$ for all $f \in W^{1,p}(\bbR^d)$; see \cite[Lemma~7.6]{GilbargTrudinger2001}. In other words, the Sobolev space is even a sublattice of the Banach lattice $L^p(\bbR^d)$. 
In contrast, for $k \ge 2$, the space $W^{k,p}(\bbR^d)$ is not a vector lattice \cite[Example~(d) on Page~419]{ArendtNittka2009}. 
In this note, we are interested in the case of negative $k$. 
At first glance, the situation seems to be trivial: 
for an integer $k > 0$, there are elements of $W^{-k,p}(\bbR^d)$ which cannot be written as the difference of two positive elements. In particular, the space is not a vector lattice. 

However this simple answer does not really give us much insight into the order structure of $W^{-k,p}(\bbR^d)$ -- it simply means that the cone is small compared to the entire space (which is, from a theoretical point of view, a consequence of the fact that the cone in Sobolev spaces of positive order is not \emph{normal}; see the end of the introduction for the definition of this and further notions).
If we restrict our attention to the span $W^{-k,p}(\bbR^d)_+ - W^{-k,p}(\bbR^d)_+$ of the positive cone, we shall see in Theorem~\ref{exa:negative-sobolev-spaces} that this is indeed a vector lattice. 
On bounded open sets $\Omega \subseteq \bbR^d$ the situation is more subtle and we prove a partial result for this case in the rest of Section~\ref{section:sobolev}.

A related phenomenon occurs, in a more abstract setting, in the theory of positive operator semigroups. 
If $X$ is a Banach lattice and $T = (T(t))_{t \ge 0}$ is a positive $C_0$-semigroup on $X$, one can associate a so-called \emph{extrapolation space} $X_{-1}$ to $T$ -- this is a larger space into which $X$ embeds continuously and densely. 
The space $X_{-1}$ can be considered as an abstract Sobolev space of negative order.
To get an order structure on $X_{-1}$ it is natural to define the cone $X_{-1,+}$ as the closure of $X_+$ within the space $X_{-1}$. 
This order structure was introduced by Bátkai, Jacob, Wintermayr, and Voigt in \cite{BatkaiJacobVoigtWintermayr2018} in order to study positive unbounded perturbations of semigroup generators; they also proved that this is indeed a cone and that $X_+ = X_{-1,+} \cap X$, i.e., the order structures on $X$ and $X_{-1}$ are compatible \cite[Proposition~2.3]{BatkaiJacobVoigtWintermayr2018}. 
In addition to further results about positive perturbations \cite{BarbieriEngel2025, BarbieriEngel2026}, 
the order on $X_{-1}$ also plays a significant role in the theory of positive linear systems in infinite-dimensions: 
there one is interested in positive \emph{control operators} $B$ that operate from an ordered Banach space $U$, the so-called \emph{input space}, into the extrapolation space $X_{-1}$; see the recent papers \cite{Gantouh2022a, AroraGlueckPaunonenSchwenninger2024} and also \cite{Gantouh2024}.
Just as for the concrete Sobolev space $W^{k,p}(\bbR^d)$ for $k < 0$, the cone $X_{-1,+}$ does not span $X_{-1}$, in general, and hence $X_{-1}$ is not a vector lattice. 
This was observed in a concrete example in \cite[Example~5.1]{BatkaiJacobVoigtWintermayr2018}. 
Yet, we show in Section~\ref{section:extrapolation} that a similar phenomenon as for the concrete Sobolev spaces occurs, i.e., the span $X_{-1,+} - X_{-1,+}$ is indeed a vector lattice provided that the Banach lattice $X$ has order continuous norm.

Our results are based on more abstract theorems about how to transfer lattice properties between ordered Banach spaces that are not isomorphic. 
We prove those abstract theorems in Section~\ref{section:abstract}. 
Different results regarding when an ordered Banach space (or ordered vector space) is automatically a vector lattice can be found in \cite[Section~4]{Glueck2021}.

We close this introduction by recalling  few notions from the theory of ordered Banach spaces and Banach lattices.

\subsection*{Preliminaries on ordered Banach spaces and Banach lattices}

Let $X$ be a real vector space. A non-empty set $X_+\subseteq X$ is called a \emph{cone} if  $\alpha X_+ +\beta X_+ \subseteq X_+$ for all real numbers $\alpha,\beta \ge 0$ and $X_+ \cap (-X_+) = \{0\}$. The real vector space $X$ together with a cone $X_+$ is called an \emph{ordered vector space} and $X_+$ is called the \emph{positive cone} of $X$. 
The term \emph{ordered space} is justified by the fact that $X_+$ induces a partial order on $X$ that is compatible with the vector space structure: one sets~$x\le y$ for $x,y \in X$ if and only if $y-x \in X_+$. 
So in particular, $x \ge 0$ if and only if $x \in X_+$. 
For this reason, the elements of $X_+$ are called the \emph{positive} elements of $X$. 
We call the cone of $X$ \emph{generating} if $X=X_+ - X_+$.
For $x,z\in X$, the \emph{order interval} $[x,z]$ is defined as the set $\{y \in X: x\le y \le z\}$; note that $[x,z]$ is non-empty if and only if $x \le z$.
A non-empty subset $C\subseteq X_+$ is called a \emph{face} of $X_+$  if it is also a cone and if $0 \le x \le y$ in $X$ and $y \in C$ implies $x\in C$.

By an \emph{ordered Banach space}, we mean a Banach space $X$ which is also an ordered vector space whose positive cone $X_+$ is closed with respect to the norm topology. 
The cone of an ordered Banach space is called \emph{normal} if there exists a number $M\ge 1$ such that $0\le x\le y$ implies $\norm x \le M \norm y$ for all $x, y\in X$. If $X$ is an ordered Banach space, then it is easy to see that $\linSpan(X_+) = X_+ - X_+$. 
Endowing this space with the norm
\begin{equation}
    \label{eq:norm-on-span}
    \norm{x}_{\linSpan(X_+)} 
    := 
    \inf\{\norm{y} + \norm{z} : \, y,z \in X_+ \text{ and } x = y-z\}
\end{equation}
makes $\linSpan(X_+)$ an ordered Banach space and makes the embedding $\linSpan(X_+) \hookrightarrow X$ continuous \cite[Lemma~2.2]{ArendtNittka2009}. 
Moreover, on the cone $X_+$ this norm coincides with the norm on $X$. Hence, if the cone $X_+$ is normal in $X$, then it is also normal with respect to $\norm{\argument}_{\linSpan(X_+)}$.

An ordered vector space $X$ is called a \emph{vector lattice} if any two elements have a supremum (equivalently, an infimum). The positive cone of a vector lattice is always generating and each element $x\in X$ has a \emph{modulus} that is defined as $\modulus{x}:= \sup\{-x,x\}$. 
An ordered Banach space is called a \emph{Banach lattice} if it is a vector lattice and the order is compatible with the norm in the sense that
\[
    \modulus x\le \modulus y \Rightarrow \norm x \le \norm y
\]
for all $x,y \in X$.
The cone of a Banach lattice is always normal. Moreover, the elements $x^+:= \sup\{x,0\}$ and $x^-:=\sup\{-x,0\}$ satisfy $x=x^+ - x^-$, so the cone of a Banach lattice is also generating. 
If $X$ is an ordered Banach space and $\linSpan(X_+)$ is norm dense in $X$, then the so called \emph{dual wedge}  
\begin{align*}
    X'_+ 
    := 
    \{x' \in X' : \, \langle x', x \rangle \ge 0 \text{ for all } x \in X_+\}
\end{align*}
is also a cone and thus turns the dual space $X'$ into an ordered Banach space (if $\linSpan(X_+)$ is not dense in $X$, then $X'_+ \cap -X'_+$ will not be zero, so $X'$ is then only a \emph{pre-ordered Banach space}). 
If $X$ is a Banach lattice, then the dual space is a Banach lattice as well. 

A linear map $T: X \to Z$ between two ordered vector spaces $X$ and $Z$ is called \emph{bipositive} if, for all $x \in X$, one has $Tx \ge 0$ if and only if $x \ge 0$.
For ordered Banach spaces $X$ and $Z$ we denote the space of all bounded linear operators from $X$ to $Z$ by $\calL(X,Z)$ and we abbreviate $\calL(X) := \calL(X,X)$. 
A linear map $T: X \to Z$ is called \emph{positive} if $TX_+ \subseteq Z_+$. 
As a consequence of the Hahn-Banach separation theorem, this is equivalent to $T'Z'_+ \subseteq X'_+$.
By $\calL(X,Z)_+$ we denote the set of all positive operators in $\calL(X,Z)$. The identity operator on $X$ will be denoted by $\id_X$.

\section{Abstract results on ordered Banach spaces}
\label{section:abstract}

Recall that a Banach lattice $Z$ is called \emph{monotonically complete} if every increasing norm-bounded net in $Z$ (equivalently, in $Z_+$) has a supremum \cite[Definition~2.4.18]{Meyer-Nieberg1991}. 
Every $L^p$-space is monotonically complete for $p \in [1,\infty)$ and $L^\infty$-spaces over a $\sigma$-finite measure space are also monotonically complete.
Clearly, a monotonically complete Banach lattice is Dedekind complete, i.e., every order-bounded increasing net in such a space has a supremum. 
The space $c_0$ of sequences that converge to $0$ shows that the converse implication does not hold. 
A \emph{KB-space} is a Banach lattice in which every increasing norm-bounded net (equivalently, sequence) is norm-convergent.
Typical examples of KB-spaces are all reflexive Banach lattices and $L^1$-spaces (i.e., all $L^p$-spaces for $1\le p<\infty$). A subspace $I$ of a Banach lattice $Z$ is said to be a \emph{lattice ideal} if for each $x,y\in Z$, the conditions $0\le \modulus x\le \modulus y$ and $y\in I$ imply that $x\in I$.

\begin{theorem}
    \label{thm:span-lattice-dual}
    Let $X$ be an ordered Banach space with generating cone and let $Z$ be a Banach lattice.
    Assume that there exists an operator $J \in \calL(X,Z)_+$ and a sequence $(R_n) \subseteq \calL(Z,X)_+$ such that $(R_nJ) \subseteq \calL(X)$ converges to $\id_X$ with respect to the weak operator topology.
    \begin{enumerate}[\upshape (a)]
        \item 
        The space $\linSpan(X'_+) = X'_+ - X'_+$ is a monotonically complete Banach lattice with respect to a norm  equivalent to $\norm{\argument}_{\linSpan(X'_+)}$. 

        \item 
        If $X$ is reflexive, then the Banach lattice $\linSpan(X'_+)$ is a KB-space.
        
        \item 
        If the sequence $(J R_n) \subseteq \calL(Z)$ converges to $\id_Z$ with respect to the weak operator topology, then $J'$ is a lattice homomorphism from $Z'$ to $\linSpan(X'_+)$ and $J'Z'$ is a lattice ideal in $\linSpan(X'_+)$.
    \end{enumerate}
\end{theorem}


The proof of Theorem~\ref{thm:span-lattice-dual} requires some preparation that we outsource to Lemmata~\ref{lem:convergence-in-span}, \ref{lem:face-vector-lattice}, and~\ref{lem:lattice-with-normal-cone}.
Note that, if a sequence in $X_+$ converges in $X$, then it may not necessarily converge with respect to the stronger norm $\norm{\argument}_{\linSpan(X_+)}$ as well. 
However, this is the case for increasing sequences. This observation is the content of the following lemma.

\begin{lemma}
    \label{lem:convergence-in-span}
    Let $X$ be an ordered Banach space and let $(x_j)$ be an increasing net of elements in $\linSpan(X_+)$. If $(x_j)$ is norm-convergent in $X$, then its limit is also in $\linSpan(X_+)$ and the net is even convergent with respect to $\norm{\argument}_{\linSpan(X_+)}$.
\end{lemma}

\begin{proof}
    Let $(x_j)$ be norm convergent in $X$, to say $x$. Since $(x_j)$ is increasing, $x-x_j\in X_+$ for each index $j$.  
    It thus follows from the assumption $x_j \in \linSpan(X_+)$ that $x \in \linSpan(X_+)$, too. 
    As the norm on $X$ and the norm on $\linSpan(X_+)$ coincide on $X_+$, we obtain 
    \[
        \norm{x-x_j}_{\linSpan(X_+)} = \norm{x-x_j}_X \to 0
    \]
    as claimed. 
\end{proof}

As a second ingredient for the proof of Theorem~\ref{thm:span-lattice-dual} we use the following simple lemma about bipositive maps between vector lattices. 

\begin{lemma}
    \label{lem:face-vector-lattice} 
    Let $X, Z$ be vector lattices and let $J: X \to Z$ be a bipositive linear map. 
    If $JX_+$ is a face in $Z_+$, then $J$ is a lattice homomorphism and  $JX$ is a lattice ideal in $Z$.
\end{lemma}

\begin{proof}
    Let $x \in X$. 
    We have $0\le\modulus{Jx} \le J \modulus{x}$ due to the positivity of $J$. 
    Therefore, using that $JX_+$ is a face in $Z_+$, 
    we obtain a $y \in X_+$ such that $\modulus{Jx} = Jy$. 
    Hence, 
    \[
        \pm Jx \le Jy \le J \modulus{x}.
    \]
    Since $J$ is bipositive this implies that $\pm x \le y \le \modulus{x}$, so actually $y = \modulus{x}$. 
    In turn, $\modulus{Jx} = Jy = J\modulus{x}$. 
    So $J$ is indeed a lattice homomorphism.

    Consequently, $JX$ is a vector sublattice of $Z$ and by using again that $JX_+$ is a face in $Z_+$, we get that $JX$ is even a lattice ideal in $Z$.
\end{proof}

Finally, we also need the following lemma about renorming of ordered Banach spaces that are vector lattices. 
The lemma seems to be a folklore result that is known to experts in vector lattice theory and is used on various occasions in the literature, but it is hard to find a detailed proof of it. 
Since the proof is a bit less obvious than one might expect at first glance, we include the details.

\begin{lemma}
    \label{lem:lattice-with-normal-cone}
    Let $X$ be an ordered Banach space that is also a vector lattice (with respect to the same order) and assume that the cone $X_+$ is normal. 
    Then there exists an equivalent norm on $X$ that renders $X$ a Banach lattice.
\end{lemma}

\begin{proof}
    Let $\norm{\argument}$ denote the given norm on $X$. 
    Since $X_+$ is normal there exists, by the definition of normality, a real number $M \ge 1$ such that $\norm{x} \le M\norm{y}$ for all $x,y \in X$ that satisfy $0 \le x \le y$.
    Moreover, as $X$ is a vector lattice its cone is generating, so there exists a number $C \ge 1$ such that each $x \in X$ can be decomposed as $x = y-z$ for vectors $y,z \in X_+$ that satisfy $\norm{x}, \norm{z} \le C\norm{x}$; see \cite[Theorem~2.37(1) and~(3)]{AliprantisTourky2007}.
    We define
    \begin{align*}
        \blNorm{x} 
        := 
        \sup \big\{ \norm{w} : \, 0 \le w \le \modulus{x} \big\}
    \end{align*}
    for each $x \in X$. 
    The supremum is finite due to the normality of the cone. 
    We now check that $\blNorm{\argument}$ is a norm with the claimed properties.
    
    \emph{The mapping $\blNorm{\argument}$ is indeed a norm}: 
    Clearly, $\blNorm{0} = 0$. 
    If, conversely, $\blNorm{x} = 0$ for a vector $x \in X$, 
    then it follows from $\norm{\modulus{x}} \le \blNorm{x} = 0$ that $\modulus{x} = 0$, so $x = 0$.
    For each $x \in X$ one easily checks that $\blNorm{\alpha x} = \alpha \blNorm{x}$ for all $\alpha \in [0,\infty)$ and that $\blNorm{-x} = \blNorm{x}$, so absolute homogeneity also holds. 
    The triangle inequality can be checked by using the Riesz decomposition property, i.e., the equality $[0,x+y] = [0,x] + [0,y]$ for all $x,y \in X_+$, which holds in every vector lattice \cite[Proposition~II.1.6]{Schaefer1974}.
    
    \emph{The norm $\blNorm{\argument}$ is equivalent to $\norm{\argument}$}:
    Let $x \in X$. 
    We have $x = x^+ - x^-$, so it follows from $0 \le x^+ \le \modulus{x}$ and $0 \le x^- \le \modulus{x}$ that 
    \begin{align*}
        \norm{x}
        \le
        \norm{x^+} + \norm{x^-} 
        \le 
        2M \norm{\modulus{x}}
        \le 
        2M \blNorm{x}
        .
    \end{align*}
    On the other hand, as above, we can find vectors $y,z \in X_+$ that satisfy $x = y-z$ and $\norm{y}, \norm{z} \le C \norm{x}$. 
    One has $0 \le x^+ \le y$ and $0 \le x^- \le z$, so for all $w \in X$ that satisfy $0 \le w \le \modulus{x} = x^+ + x^-$ we get 
    \begin{align*}
        \norm{w}  
        \le 
        M \norm{ \modulus{x} }
        \le 
        M\big(\norm{x^+} + \norm{x^-}\big) 
        \le 
        M^2\big(\norm{y} + \norm{z}\big) 
        \le 
        2 M^2 C \norm{x}
        .
    \end{align*}
    Hence, $\blNorm{x} \le 2 M^2 C \norm{x}$.
    
    \emph{The space $X$ is a Banach lattice with respect to $\blNorm{\argument}$}:
    Since the norm $\blNorm{\argument}$ is equivalent to the original norm, 
    $X$ is still a Banach space with respect to $\blNorm{\argument}$. 
    If $x,y \in X$ satisfy $\modulus{x} \le \modulus{y}$, then 
    \begin{align*}
        \big\{ \norm{w} : \, 0 \le w \le \modulus{x} \big\} 
        \subseteq 
        \big\{ \norm{w} : \, 0 \le w \le \modulus{y} \big\}
    \end{align*}
    and hence, $\blNorm{x} \le \blNorm{y}$.
\end{proof}

Now we have all the ingredients available that we need in order to show Theorem~\ref{thm:span-lattice-dual}.

\begin{proof}[Proof of Theorem~\ref{thm:span-lattice-dual}]
    (a) 
    We divide the proof into two steps.
    
    \emph{Step 1: The span is a lattice.}
    Let $x' \in \linSpan(X'_+)$. 
    It suffices to prove that $\pm x'$ has a supremum in $\linSpan(X'_+)$.
    Since $Z$ is a Banach lattice, so is its dual space $Z'$. 
    Hence, we can take the modulus $\modulus{R_n'x'}$ in $Z'$ for each index $n$.
    We show that the sequence $\big(J' \modulus{R_n'x'} \big)$ is norm-bounded in $X'$. 
    First of all, because $(R_nJ)$ converges to $\id_X$ in the weak operator topology, the sequence $(J'R_n')$ converges to $\id_{X'}$ in the weak${}^*$-operator topology. In particular, it is bounded.
    Secondly, as $x' \in X'_+ - X'_+$, there exists a vector $y' \in X'_+$ such that $\pm x' \le y'$. 
    Therefore, $\pm R_n' x' \le R_n'y'$ in $Z'$ and hence, $J'\modulus{R_n' x'} \le J'R_n' y'$ in $X'$ for each $n$. 
    
    Since the cone in $X$ is generating, the cone in $X'$ is normal \cite[Theorem~2.42]{AliprantisTourky2007} and hence, the previous inequality together with boundedness of the sequence $(J'R_n')$ gives the norm-boundedness of the sequence $\big(J' \modulus{R_n'x'} \big)$ as well. 

    By the Banach-Alaoglu theorem, there is a weak${}^*$-convergent subnet $\big(J' \modulus{R_{n_j}'x'} \big)$ with limit $s'$ in $X'$. 
    As the cone $X'_+$ is weak${}^*$-closed, it follows that $s' \ge 0$.
    Let us show that $s'$ is the supremum of $\pm x'$ in $\linSpan(X'_+)$. 
    On one hand, it follows from the inequality $\pm J' R_{n_j}' x' \le J' \modulus{R_{n_j}'x'}$ for each $j$,  from the convergence of $(J'R_n')$ to $\id_{X'}$ in the weak${}^*$-operator topology, and from the weak${}^*$-closedness of $X'_+$, that $\pm x' \le s'$. 
    On the other hand, for every upper bound $u'$ of $\pm x'$ in $\linSpan(X'_+)$ -- repeating the above argument with $u'$ instead of $y'$ -- we get
    \[
        J' \modulus{ R_{n_j}' x' } 
        \le 
        J' R_{n_j}' u'
    \]
    in $X'$ for each index $j$.
    Employing again the convergence of $(J'R_n')$ in the weak${}^*$-operator topology and the weak${}^*$-closed\-ness of $X'_+$, it thus follows that $ s' \le u'$. 
    So $s'$ is indeed the least upper bound of $\pm x'$ in $\linSpan(X'_+)$.

    \emph{Step 2: The span is a monotonically complete Banach lattice.} 
    As noted in Step~1, the cone in $X'$ is normal.
    Hence, $X'_+$ is also normal with respect to $\norm{\argument}_{\linSpan(X'_+)}$.
    Therefore, $\linSpan(X'_+)$ is an ordered Banach space with a normal cone and a vector lattice at the same time. 
    Hence, Lemma~\ref{lem:lattice-with-normal-cone} shows that there exists an equivalent norm that turns $\linSpan(X'_+)$ into a Banach lattice. 
    To show that $\linSpan(X'_+)$ is monotonically complete, consider an increasing net $(x'_j)$ in $X'_+$ that is norm-bounded in $\linSpan(X'_+)$. We need to show that it has a supremum in $\linSpan(X'_+)$. 
    The net $(x'_j)$ is in particular norm-bounded in $X'$ and hence, by Banach-Alaoglu, it has a subnet that converges weak${}^*$ to $x'\in X'_+$. Actually, since $(x'_j)$
    is increasing, it even weak${}^*$-converges itself to $x' \in X'_+$. 
    It is easy to check that $x'$ is the supremum of $(x'_j)$ in $X'$ and hence in $\linSpan(X'_+)$. 

    (b)
    If $X$ is reflexive, so is $X'$. Moreover, as noted above the cone of $X'_+$ is normal. Therefore, every increasing norm bounded net in $X'_+$  is norm convergent in $X'$ \cite[Theorem~2.45]{AliprantisTourky2007}. Applying Lemma~\ref{lem:convergence-in-span}, we immediately obtain that every increasing norm bounded net in $X'_+$ is also norm convergent with respect to $\norm{\argument}_{\linSpan(X'_+)}$. In other words, $\linSpan(X'_+)$ is a KB-space.
    
    (c) 
    We begin with the preliminary observations that $J$ is injective and bipositive and that $J'$ is bipositive as well. 
    Indeed, first note that the bipositivity of $J$ holds since
    if $x \in X$ and $Jx \ge 0$, then $x$ is the weak limit of the sequence $(R_n J x) \subseteq X_+$ and hence, $x \ge 0$. 
    Being a bipositive bounded operator between ordered Banach spaces, $J$ is injective \cite[Proposition~2.1(a)]{AroraGlueckPaunonenSchwenninger2024}.
    To see the bipositivity of $J'$, note that 
    due to the assumption of~(c), the net $(R_n' J') \subseteq \calL(Z')$ converges to $\id_{Z'}$ with respect to the weak${}^*$-operator topology. 
    So if $z' \in Z'$ and $J'z' \ge 0$, then $z'$ is positive since it is the weak${}^*$-limit of $(R_n'J'z')$ and since $Z'_+$ is weak${}^*$-closed.

    Next, we show that $J'(Z'_+)$ is a face in $X'_+$; 
    as by Lemma~\ref{lem:face-vector-lattice} and the bipositivity of $J'$ this implies~(c).
    We roughly follow the arguments recently given in \cite[Theorem~2.5]{AroraGlueckPaunonenSchwenninger2024} where a related result was shown.
    Let $z' \in Z'_+$ and $x' \in X'_+$ satisfy $0 \le x' \le J' z'$; we have to find a functional $y' \in Z'_+$ that satisfies $J'y' = x'$ or, in other words, $\duality{y'}{Jx} = \duality{x'}{x}$ for all $x \in X$.
    
    To this end consider the linear map $\varphi: JX \to \bbR$ that is given by $\varphi(z) = \duality{x'}{J^{-1}z}$ for all $z \in JX$. 
    Observe that $\varphi$ is well-defined since $J$ is injective and it is positive since $J$ is bipositive and $x'$ is positive. 
    We intend to extend $\varphi$ to a functional $y' \in Z'_+$; 
    for such an extension $y'$ to exist it is, according to the extension result in \cite[Proposition~2.4]{AroraGlueckPaunonenSchwenninger2024}, sufficient for $\varphi(v) \le \langle z', w\rangle$ to hold for all $v \in JX$ and all $w \in Z_+$ that satisfy $v \le w$. 
    So consider such $v$ and $w$ and observe that indeed 
    \begin{align*}
        \varphi(v) 
        & = 
        \duality{x'}{J^{-1}v} 
        = 
        \lim_n \duality{x'}{R_n v} 
        \\
        & \le 
        \limsup_n \duality{x'}{R_n w}
        \le 
        \limsup_n \duality{J'z'}{R_n w}
        = 
        \duality{z'}{w}
        ;
    \end{align*}
    the second equality uses the weak operator convergence of $(R_n J)$ to $\id_X$,  
    the first inequality uses the positivity of $x'$, the second inequality the positivity of $w$, 
    and the last equality the weak operator convergence of $(J R_n)$ to $\id_Z$. 
    So \cite[Proposition~2.4]{AroraGlueckPaunonenSchwenninger2024} gives us a functional $y' \in Z'_+$ that extends $\varphi$. 
    The definition of $\varphi$ now readily gives that $J'y' = x'$, so $J'Z'_+$ is indeed a face in $X'_+$.
\end{proof}

\begin{remark}
    From the proof of Theorem~\ref{thm:span-lattice-dual}(b), we see that the assumption that $X$ is reflexive in Theorem~\ref{thm:span-lattice-dual}(b) can be replaced with the weaker assumption that every increasing norm bounded net in $X'_+$ is norm convergent. 
    
    This assumption is, for instance, satisfied if $X_+$ has non-empty interior, since then there exists an equivalent norm on $X'$ that is additive on the positive cone, which in turn gives that every increasing norm bounded net in $X'_+$ is a Cauchy net.
\end{remark}

Let us also explicitly state a version of Theorem~\ref{thm:span-lattice-dual} where the conclusion refers to the spaces themselves rather than to their duals. 
Due to the dualization, the roles of $X$ and $Z$ in the assumptions of the following corollary are swapped compared to Theorem~\ref{thm:span-lattice-dual}.

\begin{corollary}
    \label{cor:span-reflexive}
    Let $X$ be a Banach lattice and let $Z$ be a reflexive ordered Banach space with a normal cone. 
    Assume that there exists an operator $J \in \calL(X,Z)_+$ and a sequence $(R_n) \subseteq \calL(Z,X)_+$ such that $(JR_n) \subseteq \calL(Z)$ converges to $\id_Z$ with respect to the weak operator topology.
    Then $\linSpan(Z_+) = Z_+ - Z_+$ is a KB-space with respect to a norm equivalent to $\norm{\argument}_{\linSpan(Z_+)}$.
\end{corollary}

\begin{proof}
    The dual space $\tilde Z := X'$ is a Banach lattice. Moreover, for each $z \in Z$ we have $J(R_nz)^+ - J(R_nz)^- \to z$ weakly, so the span of the cone $Z_+$ is dense in $Z$. Hence, the dual wedge $\tilde X_+ = Z'_+$ is actually a cone in $\tilde X:= Z'$ and therefore $\tilde X$ is an ordered Banach space. Note that 
    $\tilde X_+$ is generating because the cone in $Z$ is normal \cite[Corollary~2.43]{AliprantisTourky2007}. 
    
    Further, the sequence $( R_n'J') \subseteq \calL(Z')$ converges to $\id_{Z'}$ with respect to the weak${}^*$-operator topology. 
    Since $Z'$ is reflexive, this topology coincides with the weak operator topology. 
    Applying Theorem~\ref{thm:span-lattice-dual}(b) to the spaces $\tilde X$ and $\tilde Z$ and the operators $J'$ and $R_n'$, we conclude $\linSpan(\tilde X'_+) = \linSpan(Z_+)$ is a KB-space.
\end{proof}

Even though surprising, $\linSpan(Z_+)$ need not be reflexive in Corollary~\ref{cor:span-reflexive}; see Example~\ref{exa:neumann-laplace}. 
This also shows that $\linSpan(X'_+)$ need not be reflexive in Theorem~\ref{thm:span-lattice-dual} even if $X$ is so.

\begin{remark}
    In the situation of Corollary~\ref{cor:span-reflexive} assume in addition that $X$ is reflexive and that $R_n J \to \id_X$ with respect to the weak operator topology. 
    Then the sequence $(J' R_n') \subseteq \calL(X')$ also converges to $\id_{X'}$ with respect to the weak operator topology. 
    So it follows from Theorem~\ref{thm:span-lattice-dual}(c) that $J$ is a lattice homomorphism from $X'' = X$ to $\linSpan(Z_+)$ and that $JX$ is a lattice ideal in $\linSpan(Z_+)$.
\end{remark}

A Banach lattice $X$ is said to have \emph{order continuous norm}, if every increasing and order bounded net (equivalently, sequence) in $X_+$, is norm convergent. 
Classical examples are $L^p$-spaces for $p \in [1,\infty)$ as well as the space $c_0$.
KB-spaces have order continuous norm, but the converse is not true as the example $c_0$ shows.  We refer to \cite{Wnuk1999} for a thorough and extensive presentation of the theory of Banach lattices with order continuous norms.

\begin{theorem}
    \label{thm:span-lattice-oc} 
    Let $X$ be a Banach lattice with order continuous norm and let $Z$ be an ordered Banach space. 
    Assume that there exists a positive bijection $S \in \calL(Z,X)$, a positive operator $J \in \calL(X,Z)$ and a sequence of positive operators $(R_n) \subseteq \calL(Z,X)$ such that $(JR_n) \subseteq \calL(Z)$ converges to $\id_Z$ with respect to the strong operator topology. 
    \begin{enumerate}[\upshape (a)]
        \item 
        The space $\linSpan(Z_+)$ is a Banach lattice with order continuous norm with respect to a norm that is equivalent to $\norm{\argument}_{\linSpan(Z_+)}$. 

        \item 
        If $X$ is a KB-space, then so is $\linSpan(Z_+)$.
        
        \item 
        If the sequence $(R_nJ) \subseteq \calL(X)$ converges to $\id_X$ with respect to the weak operator topology, then $J$ is a lattice homomorphism from $X$ to $\linSpan(Z_+)$ and $JX$ is a lattice ideal in $\linSpan(Z_+)$.
    \end{enumerate}
\end{theorem}

Note that we do not assume $S$ to be bipositive -- i.e., the inverse operator $S^{-1} \in \calL(X,Z)$ need not be positive. 
This will be important for our application to extrapolation spaces in Section~\ref{section:extrapolation}.

We point out that Theorem~\ref{thm:span-lattice-oc} gives that $\linSpan(Z_+)$ is a Banach lattice if $X$ is a Banach lattice with order continuous norm, while Corollary~\ref{cor:span-reflexive} requires $X$ to be reflexive to get the same conclusion. 
Similarly, Theorem~\ref{thm:span-lattice-oc} gives that $\linSpan(Z_+)$ is a KB-space  if $X$ is so, while Corollary~\ref{cor:span-reflexive} again requires $X$ to be reflexive to get the same conclusion.
The price that one pays for this weaker assumption in Theorem~\ref{thm:span-lattice-oc} is that one also needs the existence of the bijective positive operator $S$ and that the sequence $(JR_n)$ is required to converge with respect to the strong instead of the weak operator topology.

\begin{proof}[Proof of Theorem~\ref{thm:span-lattice-oc}]
    (a) 
    We prove~(a) in three steps. 
    
    \emph{Step 1: We first make the following auxiliary observation: }
    If $(f_n)$ and $(g_n)$ are two sequences in $Z$ such that 
    $0 \le f_n \le g_n$ for all indices $n$ and if $(g_n)$ converges with respect to the norm in $Z$ to a vector $g \in Z$, then $(f_n)$ has a subsequence that converges weakly in $Z$. 
    
    Indeed, after replacing sequences with subsequences, we may assume that one has $\sum_{n=1}^\infty \norm{g_n - g}_Z < \infty$.
    Then the values $\norm{Sg_n-Sg}_X$ are also summable, so the series $h := \sum_{n=1}^\infty \modulus{Sg_n - Sg}$ converges absolutely in $X$. 
    One gets that $0 \le Sf_n \le Sg_n \le h+Sg$ for all indices $n$, where the first two inequalities follow from the positivity of $S$ and the last inequality from the definition of $h$.
    So the sequence $(Sf_n)$ is contained in an order interval of $X$ and since $X$ has order continuous norm, all order intervals in $X$ are weakly compact \cite[Theorem~1.12]{Wnuk1999}.
    Thus $(Sf_n)$ has a subsequence that converges weakly in $X$. 
    Since the inverse operator $S^{-1}:X\to Z$ (which might not be positive) is continuous by the bounded inverse theorem and is thus weakly continuous from $X$ to $Z$, it follows that $(f_n)$ does indeed have a subsequence that converges weakly in $Z$.

    \emph{Step 2: The space $\linSpan(Z_+)$ is a lattice.}
    To see this, let $z \in \linSpan(Z_+)$. 
    It suffices to show that $\pm z$ has a supremum in $\linSpan(Z_+)$. As $z \in  Z_+- Z_+$, there exists $y \in Z_+$ such that $\pm z \le y$.
    Then we have $\pm R_n z \le R_n y$ and hence $0 \le J\modulus{R_n z} \le J R_n y$ for all indices $n$.
    By assumption, $(JR_n y)$ converges to $y$ with respect to the norm in $Z$, and thus Step~1 shows that $(J \modulus{R_n z})$ has a subsequence $(J \modulus{R_{n_j} z})$ that converges weakly to an element $s \in Z$. 
    As $Z_+$ is closed and convex it is weakly closed, so one has $s \in Z_+$. 
    
    We show that $s$ is the supremum of $\pm z$ within $\linSpan(Z_+)$. 
    On the one hand, it follows from $\pm J R_{n_j} z \le J \modulus{R_{n_j} z}$ for all $j$, from the convergence of $(J R_n)$ to $\id_Z$ in the strong operator topology, and from the weak closedness of $Z_+$,  that $\pm z \le s$.
    On the other hand, for every upper bound $u \in \linSpan(Z_+)$ of $\pm z$ -- repeating the above argument with $u$ instead of $y$ -- we get
    \[
         J\modulus{R_{n_j} z} \le J R_{n_j} u
    \]
    for all indices $j$. 
    Taking weak limits in $Z$ and using again that the cone $Z_+$ is weakly closed in $Z$ we get 
    $ s \le u$. 
    Hence, $s$ is indeed the least upper bound of $\pm z$ within $\linSpan(Z_+)$.

    \emph{Step 3: The space $\linSpan(Z_+)$ is a Banach lattice with order continuous norm.}
    We first show that the cone $Z_+$ in $Z$ is normal. 
    To see this, let $y,z \in Z$ satisfy $0 \le y \le z$. 
    Then we have $0 \le Sy \le Sz$ in $X$ and hence, as $X$ is a Banach lattice, $\norm{Sy} \le \norm{Sz}$. 
    The inverse operator $S^{-1}$ is bounded by the bounded inverse theorem and thus, 
    \[
        \norm{y} \le \norm{S^{-1}} \norm{Sy} \le \norm{S^{-1}} \norm{Sz} \le \norm{S^{-1}} \norm{S} \norm{z}.
    \]
    Therefore, $Z_+$ is normal in $Z$, and in turn, also in $\linSpan(Z_+)$; 
    see the discussion after~\eqref{eq:norm-on-span}.
    So $\linSpan(Z_+)$ is an ordered Banach space with a normal cone and, at the same time, a vector lattice. 
    According to Lemma~\ref{lem:lattice-with-normal-cone} it is thus a Banach lattice with respect to an equivalent norm.

    In order to show that the norm in $\linSpan(Z_+)$ is order continuous, let $(z_j)$ be an increasing order bounded net in $Z_+$; we need to show that it is norm convergent in $\linSpan(Z_+)$.
    The net $(S z_j)$ in $X_+$ is also increasing and order bounded, so it is norm convergent in $X$ since $X$ has order continuous norm. 
    By the bounded inverse theorem, we obtain that the net $(z_j)$ is norm convergent with respect to the norm in $Z$. As $(z_j)$ was increasing, we may apply Lemma~\ref{lem:convergence-in-span} to assert the claim.

    (b)
    This is essentially the same argument as in the previous paragraph -- one just argues with norm-bounded instead of order-bounded nets now.
    
    (c) 
    As the Banach lattice $X$ has order continuous norm, it is an ideal in its bidual $X''$ \cite[Theorem~II.5.10]{Schaefer1974} and thus, $X_+$ is a face in $X''_+$. 
    Moreover, it follows from the convergence $J R_n \to \id_Z$ in the strong operator topology that $JX$ is dense in $Z$. 
    Due to those properties and due to the convergence $R_n J \to \id_X$ in the weak operator topology,  \cite[Theorem~2.5]{AroraGlueckPaunonenSchwenninger2024} is applicable and gives that $JX_+$ is a face in $Z_+$ and thus in the cone of $\linSpan(Z_+)$ since the latter coincides with $Z_+$.

    The convergence $R_n J \to \id_X$ in the weak operator topology also implies that $J$ is bipositive. 
    Hence, Lemma~\ref{lem:face-vector-lattice} shows that $J$ is a lattice homomorphism from $X$ to $\linSpan(Z_+)$ and that $JX$ is a lattice ideal in $\linSpan(Z_+)$.
\end{proof}

\section{Sobolev spaces of negative order}
\label{section:sobolev}

Let $p,q \in (1,\infty)$ be Hölder conjugates and let $d,k \ge 1$ be integers. 
We first consider Sobolev spaces on the whole space $\bbR^d$ (Theorem~\ref{exa:negative-sobolev-spaces}). 
Later on, we discuss bounded domains with continuous boundaries (Theorem~\ref{thm:negative-sobolev-spaces-domain}).
We endow the Sobolev space $W^{k,p}(\bbR^d)$ with the usual order  inherited from $L^p(\bbR^d)$. 
This is an ordered Banach space with generating cone (even for $p \in [1,\infty]$, although the result is much more difficult to show for $p=1$; see the discussion in the introduction of \cite{PonceSpector2020}). 
The cone is not normal as can easily be seen by considering rapidly oscillating functions. 
Furthermore,  $W^{k,p}(\bbR^d)$ is a vector lattice if and only if $k = 1$, see, for instance, \cite[Lemma~7.6]{GilbargTrudinger2001} and \cite[Example~(d) on Page~419]{ArendtNittka2009}; however, even for the case $k=1$ the space is not a Banach lattice due to the cone not being normal. 

Let us now use Theorem~\ref{thm:span-lattice-dual} to demonstrate that the situation is different for Sobolev spaces of negative order. 
The space $W^{-k,p}(\bbR^d)$ is the dual space of $W^{k,q}(\bbR^d)$ and it thus contains $L^p(\bbR^d)$ in a canonical way. 
We endow $W^{-k,p}(\bbR^d)$ with the wedge $W^{-k,p}(\bbR^d)_+$ that is defined to be the dual wedge of the cone of $W^{k,q}(\bbR^d)$; this dual wedge is actually a cone since $W^{k,q}(\bbR^d)_+$ is generating. 
Moreover, the dual cone $W^{-k,p}(\bbR^d)_+$ coincides with the closure of $L^p(\bbR^d)_+$ in $W^{-k,p}(\bbR^d)$; 
this follows from from reflexivity of $W^{-k,p}(\bbR^d)$ and from the bipositivity of the embedding of $W^{k,q}(\bbR^d)$ into $L^q(\bbR^d)$. 
As noted above, the cone $W^{k,q}(\bbR^d)_+$ is not normal. 
Hence, by duality the cone $W^{-k,p}(\bbR^d)_+$ is not generating in $W^{-k,p}(\bbR^d)$ \cite[Theorem~2.40]{AliprantisTourky2007}. 
So in particular $W^{-k,p}(\bbR^d)$ cannot be a vector lattice. 
Nevertheless, we show that the span of the cone in this space is still a vector lattice -- and even a Banach lattice when endowed with an appropriate norm. 

\begin{theorem}
    \label{exa:negative-sobolev-spaces}
    Let $p \in (1,\infty)$ and let $d,k \ge 1$ be integers.
    The span of $W^{-k,p}(\bbR^d)_+$ in $W^{-k,p}(\bbR^d)$ is a KB-space (in particular, a Banach lattice) with respect to a norm equivalent to $\norm{\argument}_{\linSpan(W^{-k,p}(\bbR^d)_+)}$. 
    Moreover, $L^p(\bbR^d)$ is a lattice ideal in $\linSpan(W^{-k,p}(\bbR^d)_+)$ and the canonical embedding is a lattice homomorphism.
\end{theorem}

\begin{proof}
    Let $q \in (1,\infty)$ be the Hölder conjugate of $p$ and
    let $J: X := W^{k,q} (\bbR^d) \to L^{q}(\bbR^d) =: Z$ be the canonical embedding. 
    The cone in $X$ is generating (see the introduction of \cite{PonceSpector2020}).
    Let $\rho: \bbR^d \to [0,\infty)$ be a test function with integral $1$ and define the mollifier sequence $(\rho_n)$ by setting $\rho_n := n^d \rho(n \argument)$ for each integer $n \ge 1$. 
    For each $n$, let $R_n: Z \to X$ be given by $R_n f = \rho_n \star f$ for every $f \in Z$. 
    Then $\big(R_n J\big)$ converges strongly to the identity on $X$ \cite[Proof of Theorem~2.1.2]{Kesavan1989} and $\big(J R_n\big)$ converges strongly to the identity on $Z$ \cite[Corollary~1.5.2]{Kesavan1989}.
    Hence, the claim follows by an application of Theorem~\ref{thm:span-lattice-dual} since $W^{-k,p}(\bbR^d) = X'$.
\end{proof}

Let us now consider the case of domains -- or, more generally, open subsets -- in $\bbR^d$. 
Let $p \in (1,\infty)$, let $d,k \ge 1$ be integers, and let $\Omega \subseteq \bbR^d$ be a non-empty and open. 
Recall that $W^{k,q}_0(\Omega)$ is the closure of all test functions on $\Omega$ within the Sobolev space $W^{k,q}(\Omega)$; here $q$ denotes again the Hölder conjugate of $p$.
The space $W^{k,q}_0(\Omega)$ is an ordered Banach space with respect to the order inherited from $L^q(\Omega)$. 
The span of the cone in $W^{k,q}_0(\Omega)$ contains all test functions on $\Omega$ (since each such test function is the difference of two positive test functions on $\Omega$) and hence, the span of the cone is norm dense in $W^{k,q}_0(\Omega)$. 
It is, however, not clear under which conditions the cone in this space is even generating -- see Proposition~\ref{prop:generating-cone-1D}, the discussion that precedes this proposition, and Open Problem~\ref{op:generating}.

The dual space of $W^{k,q}_0(\Omega)$ is denoted by $W^{-k,p}(\Omega)$. 
As the span of the cone of $W^{k,q}_0(\Omega)$ is norm dense, the dual space $W^{-k,p}(\Omega)$ is also an ordered Banach space with respect to the dual cone. 
The dual cone coincides with the closure of $L^p(\Omega)_+$ in $W^{-k,p}(\Omega)$ under the canonical embedding.

Our goal is to show that the span of the cone in $W^{-k,p}(\Omega)$ is a lattice if the boundary of $\Omega$ is sufficiently nice. 
To be precise, let us recall the following notions. 
A map $\bbR^d \to \bbR^d$ is called a \emph{rigid motion} if it is the composition of a rotation and a translation. 
The set $\Omega$ is said to have \emph{continuous boundary} if for every $\bm x_0 \in \partial \Omega$, there exists a radius $r > 0$, a rigid motion $M: \bbR^d \to \bbR^d$, and a continuous function $F: \bbR^{d-1} \to \bbR$ such that $M(\bm x_0) = \bm 0$ and 
\begin{align*}
    M \big(\Omega \cap \openBall{r}{\bm x_0}\big) 
    = 
    \big\{\bm  x \in \openBall{r}{\bm 0} : \, F(x_1, \dots, x_{d-1}) < x_d \big\}
    ,
\end{align*}
where $\openBall{r}{\bm c}$ denotes the open Euclidean ball in $\bbR^d$ with radius $r$ and center $\bm c$;
see for instance \cite[Definition~9.57]{Leoni2017}. 
Observe that $F(\bm 0) = 0$: 
On one hand, since $\bm x_0$ is not an element of $\Omega \cap \openBall{r}{\bm x_0}$ and in turn, $\bm 0 = M\bm (\bm x_0)$ is not an element of $M \big(\Omega \cap \openBall{r}{\bm x_0}\big)$, one has $F(\bm 0) \ge 0$.
On the other hand, as $\bm x_0$ is a boundary point of $\Omega$, there exists a sequence in $\Omega \cap \openBall{r}{\bm x_0}$ that converges to $\bm x_0$; 
by applying $M$ to this sequence we get a sequence in the set $M\big(\Omega \cap \openBall{r}{\bm x_0}\big)$ that converges to $\bm 0$, so $F(\bm 0) \le 0$.

Note that the definition of continuous boundary also makes sense in the case $d=1$ (if one uses the convention $\bbR^0 := \{0\})$. 
It is easy to check that a non-empty bounded open set $\Omega \subseteq \bbR$ has continuous boundary if and only if it is the disjoint union of finitely many open intervals that have a non-zero distance to each other.

\begin{theorem}
    \label{thm:negative-sobolev-spaces-domain}
    Let $p,q \in (1,\infty)$ be Hölder conjugates, let $d,k \ge 1$ be integers, and let $\Omega \subseteq \bbR^d$ be a non-empty bounded open set with continuous boundary. 
    Assume that the cone in $W^{k,q}_0(\Omega)$ is generating.
    
    Then the span of $W^{-k,p}(\Omega)_+$ in $W^{-k,p}(\Omega)$ is a KB-space (in particular, a Banach lattice) with respect to a norm equivalent to $\norm{\argument}_{\linSpan(W^{-k,p}(\Omega)_+)}$. 
    Moreover, $L^p(\Omega)$ is a lattice ideal in $\linSpan(W^{-k,p}(\Omega)_+)$ and the canonical embedding is a lattice homomorphism.
\end{theorem}

We discuss the assumption that $W_0^{k,q}(\Omega)$ has a generating cone at the end of this section; see in particular the discussion after the proof of Theorem~\ref{thm:negative-sobolev-spaces-domain} as well as Proposition~\ref{prop:generating-cone-1D} and Open Problem~\ref{op:generating}. 

For the proof of Theorem~\ref{thm:negative-sobolev-spaces-domain}, we once again use Theorem~\ref{thm:span-lattice-dual}. 
To construct the operators $R_n: L^q(\Omega) \to W^{k,q}_0(\Omega)$ we can now, in contrast to the situation on the whole space, not simply convolute with a sequence of mollifiers since we need to take care of the zero boundary conditions. 
To this end, we first perturb all functions in $L^q(\Omega)$ a bit to push them away from the boundary; 
this is the content of the following two lemmas. 
The arguments are inspired by common approximation arguments in Sobolev spaces -- see for instance in the proof of \cite[Theorem~11.35]{Leoni2017} -- but we need to make some adaptations to ensure that the approximation process depends linearly on the function that is approximated.

Recall that a mapping $A \colon \bbR^d \to \bbR^d$ is called \emph{affine} if there exists a matrix $B \in \bbR^{d \times d}$ and a vector $\bm b \in \bbR^d$ such that $A(\bm x) = B\bm x +\bm  b$ for all $\bm x \in \bbR^d$. 

\begin{lemma}
    \label{lem:a-bit-of-space-pointwise} 
    Let $d \ge 1$ be an integer, let $\Omega \subseteq \bbR^d$ be a non-empty open set with continuous boundary, and let $\bm x_0 \in \overline{\Omega}$. 
    Then there exists a bounded open neighbourhood $V \subseteq \bbR^d$ of $\bm x_0$ and a sequence of bijective affine mappings $A_n \colon \bbR^d \to \bbR^d$, $\bm x \mapsto B_n \bm x +\bm  b_n$, with the following properties: 
    \begin{enumerate}[\upshape (a)]
        \item 
        One has $B_n \to \id$ and $\bm b_n \to 0$ as $n \to \infty$.

        \item 
        One has $\overline{A_n(\Omega \cap V)} \subseteq \Omega$ for every $n$.
    \end{enumerate}
\end{lemma}

\begin{proof}
    We distinguish between two cases.

    \emph{First case:} $\bm x_0 \in \Omega$. 
    Choose $V := \openBall{r}{\bm x_0}$ for a number $r > 0$ that is sufficiently small to ensure that $V \subseteq \Omega$. 
    For each integer $n \ge 2$, let $A_n$ be the compression with factor $1-\frac{1}{n}$ and center $\bm x_0$, i.e.\ let $A_n(\bm x) := (1-\frac{1}{n}) (\bm x-\bm x_0)+\bm x_0$ for all $\bm x \in \bbR^d$. 
    In other words, we set $B_n := (1-\frac{1}{n})\id \in \bbR^{d \times d}$ and $b_n := \frac{1}{n}\bm x_0 \in \bbR^d$.
    Then~(a) and~(b) are clearly satisfied. 
    The bijectivity of $A_n$ holds since $n \ge 2$.

    \emph{Second case:} $\bm x_0 \in \partial \Omega$. 
    This is the interesting case. 
    Since $\Omega$ has a continuous boundary we may, after applying a rigid motion, assume that $\bm x_0 = 0$ and that there exists a number $r > 0$ and a continuous function $F: \bbR^{d-1} \to \bbR$ such that
    \begin{align*}
        \Omega \cap \openBall{r}{\bm 0} 
        = 
        \big\{ \bm x \in \openBall{r}{\bm 0} : \, F(x_1, \dots, x_{d-1}) < x_d \big\},
    \end{align*}
    and $F(\bm 0) = 0$. 
    By the continuity of $F$ we can find a number $\delta \in (0,\frac{r}{8})$ such that $F(\bm z) \le \frac{r}{8}$ for all $\bm z \in \bbR^{d-1}$ of norm $\norm{\bm z} \le \delta$. 
    Let $\bm c := (0,\dots,0,\frac{r}{4}) \in \Omega \cap \openBall{r}{0}$ and set $V := \openBall{\delta}{\bm 0} \times (-\frac{r}{4}, \frac{3r}{4}) \subseteq \bbR^d$, where the ball $\openBall{\delta}{\bm 0}$ in the Cartesian product is taken in $\bbR^{d-1}$.
    By using that $\delta < \frac{r}{8}$ one can check that $V \subseteq \overline{V} \subseteq \openBall{r}{\bm 0}$. 
    Moreover, both vectors $\bm x_0=0$ and $\bm c$ are located in $V$.

    Now, for each integer $n \ge 2$, let $B_n \in \bbR^{d \times d}$ be the diagonal matrix whose first $d-1$ diagonal entries are $1$ and whose last diagonal entry is $1-\frac{1}{n}$, and define
    \begin{align*}
        A_n(\bm x) 
        := 
        B_n(\bm x-\bm c)+\bm c 
        = 
        B_n \bm x + (\id-B_n\bm )\bm c 
        = 
        B_n \bm x + \frac{1}{n} \bm c
    \end{align*}
    for all $\bm x \in \bbR^d$ -- i.e.\ $A_n$ is a compression in the direction of the $d$-th axis with factor $1-\frac{1}{n}$ and center $\bm c$. 
    Since $n \ge 2$, the matrix $B_n$ is invertible and hence, $A_n$ is bijective.
    Note that applying the mapping $A_n$ to a vector only changes the $d$-th coordinate of the vector, but not the first $d-1$ coordinates.
    
    Let us show that the properties~(a) and~(b) are satisfied: 

    (a) 
    One clearly has $B_n \to \id$ and $\bm b_n := \frac{1}{n} \bm c \to \bm 0$ has $n \to \infty$.
    
    (b) 
    Fix an integer $n \ge 2$, let $\bm y \in \overline{A_n(\Omega \cap V)} = A_n(\overline{\Omega \cap V})$ and choose $\bm x \in \overline{\Omega \cap V}$ such that $A_n(\bm x) =\bm  y$. 
    We note that $(x_1, \dots, x_{d-1}) = (y_1, \dots, y_{d-1})$. 
    It is clear from geometric considerations that $V$ is invariant under $A_n$ and hence, the same is true for $\overline{V}$. 
    Thus, $\bm y = A_n(\bm x) \in \overline{V} \subseteq \openBall{r}{\bm 0}$. 
    Let us now show that even $\bm y \in \Omega \cap \openBall{r}{\bm 0}$, which will conclude the proof. 
    It suffices to prove that $F(y_1, \dots, y_{d-1}) < y_d$, and to this end we distinguish between the following two cases:

    If $x_d \ge \frac{r}{4}$, then it follows from the definition of $A_n$ that also $y_d \ge \frac{r}{4}$. 
    Since $\bm x \in \overline{V}$, one has $\norm{(x_1, \dots, x_{d-1})} \le \delta$, so the choice of $\delta$ implies that also $F(y_1, \dots, y_{d-1}) = F(x_1, \dots, x_{d-1}) \le \frac{r}{8} < \frac{r}{4} \le y_d$. 
    
    If $x_d < \frac{r}{4}$, then the definition of $A_n$ implies that $x_d < y_d$. 
    Since $\bm x \in \overline{\Omega \cap V} \subseteq \overline{\Omega \cap \openBall{r}{\bm 0}}$, we have $F(y_1, \dots, y_{d-1}) = F(x_1, \dots, x_{d-1}) \le x_d < y_d$.
\end{proof}

\begin{lemma}
    \label{lem:a-bit-of-space}
    Let $p \in [1,\infty)$, let $d,k \ge 1$ be integers, and let $\Omega \subseteq \bbR^d$ be a non-empty open bounded set with continuous boundary. 
    There exists a sequence of positive linear operators $(S_n)$ on $L^p(\Omega)$ and a sequence of compact subsets $K_n \subseteq \Omega$ with the following properties:
    \begin{enumerate}[\upshape(a)]
        \item 
        The operators $S_n$ converge strongly to the identity on $L^p(\Omega)$.

        \item 
        Each operator $S_n$ leaves $W^{k,p}_0(\Omega)$ invariant and the restrictions of the $S_n$ to this space converge strongly to the identity on $W^{k,p}_0(\Omega)$. 

        \item 
        For each $n$ and each $f \in L^p(\Omega)$, the function $S_n f$ vanishes outside of $K_n$.
    \end{enumerate}
\end{lemma}

\begin{proof}
    For each $\bm y \in \overline{\Omega}$ there exists, according to Lemma~\ref{lem:a-bit-of-space-pointwise}, a bounded open neighbourhood $V_{\bm y} \subseteq \bbR^d$ of $\bm y$ and a sequence of affine bijections $A_{{\bm y},n} \colon \bbR^d \to \bbR^d$, $\bm x \mapsto B_{{\bm y},n}\bm x + b_{{\bm y},n}$ such that $B_{{\bm y},n} \to \id$ and $b_{{\bm y},n} \to 0$ as $n \to \infty$ and such that $\overline{A_{{\bm y},n}(\Omega \cap V_{\bm y})} \subseteq \Omega$ for every $n$. 
    Since the open sets $V_{\bm y}$ cover the compact set $\overline{\Omega}$, we can find a finite set of points $\emptyset \not= Y \subseteq \overline{\Omega}$ such that $\bigcup_{\bm y \in Y} V_{\bm y} \supseteq \overline{\Omega}$. 
    Define $K_n := \bigcup_{\bm y \in Y}  \overline{A_{{\bm y},n}(\Omega \cap V_{\bm y})}$ for every $n$. 
    Then each $K_n$ is a compact subset of $\Omega$. 

    For each $\bm y \in Y$ and each index $n$ we define a bounded linear operator $T_{{\bm y},n}: L^p(\bbR^d) \to L^p(\bbR^d)$ by $T_{{\bm y},n}f = f \circ A_{{\bm y},n}^{-1}$. 
    Then the operator $T_{{\bm y},n}$ restricts to a bounded linear operator on $W^{k,p}(\bbR^d)$ and, since $p < \infty$, one can check that for each $\bm y \in Y$ the operators $T_{{\bm y},n}$ converge to the identity operator strongly on $L^p(\bbR^d)$ and on $W^{k,p}(\bbR^d)$ as $n \to \infty$; 
    this is where one uses that $B_{{\bm y},n} \to \id$ and $b_{{\bm y},n} \to 0$.
    
    Now we use the existence of a smooth partition of unity \cite[Theorem~C.21 and Exercise~C.22]{Leoni2017}:
    there exist test functions $h_{\bm y} \in C^\infty_c(\bbR^d)$ for $\bm y \in Y$ that map into $[0,1]$ and have the following properties: the closed support of each $h_{\bm y}$ is contained in $V_{\bm y}$ and $\sum_{y \in Y} h_{\bm y}(\bm x) = 1$ for all $\bm x \in \overline{\Omega}$. 
    For each $n$, we now define the operator $S_n: L^p(\bbR^d) \to L^p(\bbR^d)$ by
    \begin{align*}
        S_n:  f \mapsto \sum_{\bm y \in Y} T_{{\bm y},n}(fh_{\bm y}).
    \end{align*}
    Then each operator $S_n$ leaves $W^{k,p}(\bbR^d)$ invariant and the sequence $(S_n)$ converges strongly -- both on $L^p(\bbR^d)$ and on $W^{k,p}(\bbR^d)$ -- to the operator 
    \begin{align*}
        S: f \mapsto f \sum_{\bm y \in Y} h_{\bm y}.
    \end{align*} 
    From now on we consider $C^\infty_c(\Omega)$ as a subspace of $C^\infty_c(\bbR^d)$ by extending each function with the value $0$ outside of $\Omega$. 
    Similarly, we consider $L^p(\Omega)$ as a subspace of $L^p(\bbR^d)$. 
    Then $S$ leaves $L^p(\Omega)$ invariant and acts as the identity operator on this space.
    
    We now observe that, for every $f \in C^\infty_c(\Omega)$ and each $n$, the support of $S_nf$ is contained in $K_n$. 
    To see this, fix such an $f$ and an index $n$ and let $\bm x \in \bbR^d$ such that $(S_nf)(\bm x) \not= 0$. 
    Then there exists a $\bm y \in Y$ such that $T_{{\bm y},n}(fh_{\bm y})(\bm x) \not= 0$ 
    and hence, $f\big(A_{{\bm y},n}^{-1}(\bm x)\big) \not= 0$ and $h_{\bm y}\big(A_{{\bm y},n}^{-1}(\bm x)\big) \not= 0$. 
    The first inequality implies that $A_{{\bm y},n}^{-1}(\bm x) \in \Omega$ and the second one implies that $A_{{\bm y},n}^{-1}(\bm x) \in V_{\bm y}$. 
    Thus, $\bm x \in A_{{\bm y},n}(\Omega \cap V_{\bm y}) \subseteq K_n$, as claimed. 

    So in particular, each $S_n$ leaves $C_c^\infty(\Omega)$ invariant. 
    Since each $S_n$ is continuous from $W^{k,p}(\bbR^d)$ to $W^{k,p}(\bbR^d)$ and $C_c^\infty(\Omega)$ is dense in $W^{k,p}_0(\Omega)$ with respect to the $W^{k,p}$-norm, it follows that each $S_n$ also leaves $W^{k,p}_0(\Omega)$ invariant. 
    Since $C^\infty_c(\Omega)$ is also dense in $L^p(\Omega)$, each of the operators $S_n$ is continuous with respect to the $L^p$-norm and the space of $L^p$-functions that vanish outside of $K_n$ is closed in $L^p$, it follows that $S_nf$ vanishes outside of $K_n$ for each $n$ and for each $f \in L^p(\Omega)$. 

    Finally, for each $f \in L^p(\Omega) \subseteq L^p(\bbR^d)$ we have $S_nf \to Sf = f$ with respect to the $L^p$-norm as $n \to \infty$.
    Similarly, one has $S_nf \to Sf = f$ with respect to the $W^{k,p}$-norm for all $f \in W^{k,p}_0(\Omega)$.
\end{proof}

The assumption that $\Omega$ be bounded in Lemma~\ref{lem:a-bit-of-space} (and hence in Theorem~\ref{thm:negative-sobolev-spaces-domain}) is needed since the first part of the proof uses the compactness of $\overline{\Omega}$. 
We do not know if a similar result holds for unbounded $\Omega$ (under, maybe, some kind of uniformity assumption on the continuous boundary).

\begin{proof}[Proof of Theorem~\ref{thm:negative-sobolev-spaces-domain}]
    We apply Theorem~\ref{thm:span-lattice-dual} to the spaces $X = W^{k,q}_0(\Omega)$ and $Z = L^q(\Omega)$, with $J: X \to Z$ the canonical embedding. 
    Then $J'$ is the canonical embedding of $Z' = L^p(\Omega)$ into $X' = W^{-k,p}(\Omega)$. 
    So it suffices to construct a sequence of positive operators $R_n: Z \to X$ such that $(R_n J)$ converges strongly to $\id_X$ and $(J R_n)$ converges strongly to $\id_Z$. 

    To this end, let the sequence $(S_n)$  in $\calL(L^q(\Omega))$ and the compact sets $K_n$ be as in Lemma~\ref{lem:a-bit-of-space}. (with $p$ in the lemma replaced with $q$). 
    Define $\delta_n := \dist(K_n, \partial \Omega)/3 > 0$ for each index $n$.
    Let $\rho \in C_c^\infty(\bbR^d)$ be a positive test function with integral $1$ and define the mollifiers $\rho_n := \frac{1}{\delta_n^d} \rho(\frac{1}{\delta_n}\argument)$ for each $n$. 
    For every $n$ we set $R_nf := \rho_n \star (S_nf)$ for all $f \in L^q(\Omega)$. 
    Then every operator $R_n$ maps $Z = L^q(\Omega)$ to $C^\infty_c(\Omega) \subseteq X$ and the sequences $(R_n J)$ and $(J R_n)$ converge strongly to $\id_X$ and $\id_Z$ in $\calL(X)$ and $\calL(Z)$, respectively. 
\end{proof}

We end this section with a discussion of the assumption in Theorem~\ref{thm:negative-sobolev-spaces-domain} that the cone in $W^{k,q}_0(\Omega)$ be generating. 
In the proof, this property is needed since the space $X$ in Theorem~\ref{thm:span-lattice-dual} is assumed to have a generating cone. 
In the following two cases, it is not difficult to check that the cone in $W^{k,q}_0(\Omega)$ is indeed generating. 

The first case is when $k=1$ and when the boundary of $\Omega$ is sufficiently smooth to ensure that the trace operator is well-defined on $W^{1,q}(\Omega)$ and that $W^{1,q}_0(\Omega)$ is the kernel of the trace operator. 
Indeed, it is well-known that $W^{1,q}(\Omega)$ is a vector lattice and one can then check that $W^{1,q}_0(\Omega)$ is a sublattice thereof; hence, it has a generating cone. 
However, in this case, one does not need our results to deduce that the span of the cone in the dual space $(W^{1,q}_0(\Omega))' = W^{-1,p}(\Omega)$  is a lattice. Indeed this also follows from classical results in lattice theory, for instance, \cite[Corollary~2.50(2)]{AliprantisTourky2007}.

The second case is when the dimension $d$ is equal to $1$. 
In this case, one has the following result: 

\begin{proposition}
    \label{prop:generating-cone-1D} 
    Let $I \subseteq \bbR$ be a non-empty bounded open interval, let $p \in [1,\infty)$ and let $k \ge 1$ be an integer. 
    Then the cone in $W^{k,p}_0(I)$ is generating.
\end{proposition}

\begin{proof}
    We may, and shall, assume that $I = (0,1)$. 
    Consider the ordered Banach spaces
    \[
        X_b := \{f \in W^{k,p}(I): \, f^{(j)}(b) = 0~\forall~j=0,1,\ldots,k-1\}
        \quad 
        \text{for }b \in \{0,1\}
    \]
    with the order inherited from $W^{k,p}(I)$.  
    The operator $L^p(I) \to X_0$ that integrates each function $k$ times from $0$ to the spatial variable is a positive surjection; hence, the cone in $X_0$ is generating. 
    A similar argument shows that the cone in $X_1$ is generating. 

    Now let $f \in W^{k,p}_0(I) = X_0 \cap X_1$. 
    Then by the preceding paragraph, there exist positive functions $f_0 \in X_0$ and $f_1 \in X_1$ that both dominate $f$. 
    Choose $C^\infty$-functions $u_0,u_1: [0,1] \to [0,\infty)$ such that $u_0$ is constantly $1$ on $[0,\frac{1}{2}]$ and vanishes on $[\frac{3}{4}, 1]$ and such that $u_1$ is constantly $1$ on $[\frac{1}{2},1]$ and vanishes on $[0,\frac{1}{4}]$. 
    Then $u_0 f_0 + u_1 f_1$ is in $W^{k,p}_0(I)$, is positive, and dominates $f$. 
\end{proof}

Theorem~\ref{thm:negative-sobolev-spaces-domain} and Proposition~\ref{prop:generating-cone-1D} together give the following result in one spatial dimension:

\begin{corollary}
    \label{cor:negative-sobolev-spaces-domain-1D}
    Let $p \in (1,\infty)$, let $k \ge 1$ be an integer, and let $I \subseteq \bbR$ be a non-empty bounded open interval.
    
    Then the span of $W^{-k,p}(I)_+$ in $W^{-k,p}(I)$ is a KB-space (in particular, a Banach lattice) with respect to a norm equivalent to $\norm{\argument}_{\linSpan(W^{-k,p}(I)_+)}$. 
    Moreover, $L^p(I)$ is a lattice ideal in $\linSpan(W^{-k,p}(I)_+)$ and the canonical embedding is a lattice homomorphism.
\end{corollary}

If one wants to apply Theorem~\ref{thm:negative-sobolev-spaces-domain} in dimension $d \ge 2$ and for order $k \ge 2$, this leaves the following question:

\begin{open_problem}
    \label{op:generating}
    Let $\Omega \subseteq \bbR^d$ be non-empty and open, let $p \in [1,\infty]$ and let $k \ge 2$ be an integer. 
    Under which assumptions is the positive cone in $W^{k,p}_0(\Omega)$ generating? 
\end{open_problem}

\section{Extrapolation spaces of positive semigroups}
\label{section:extrapolation}

Let $X$ be a Banach lattice and let $(T(t))_{t \ge 0}$ be a $C_0$-semigroup on $X$ that is \emph{positive}, which means that each operator $T(t)$ is positive. 
Let $A$ denote the generator of $(T(t))_{t \ge 0}$; this is a closed and densely defined operator.
For the general theory of $C_0$-semigroups, we refer to \cite{EngelNagel2000} and for positive $C_0$-semigroups specifically to \cite{Nagel1986, BatkaiKramarRhandi2017}.

Fix a real number $\lambda$ that is larger than the real part of every spectral value of $A$. 
The \emph{extrapolation space} $X_{-1}$ is defined as follows: 
for every $x \in X$ one sets $\norm{x}_{-1} := \norm{(\lambda-A)^{-1}x}_X$ and then defines $X_{-1}$ to be the completion of $X$ with respect to this norm. 
Note that $X_{-1}$ does not depend on the choice of $\lambda$ since all norms that one obtains in this way are equivalent. 
One can extend the semigroup generator $A$ to a unique operator $A_{-1}\in \calL(X, X_{-1})$. 
The resolvent can also be extended to an operator from $X_{-1}$ to $X$ that then intertwines $A$ and $A_{-1}$.

The space $X_{-1}$ plays, for instance, an important role in perturbation and linear systems theory. 
In order to study positive perturbations and positive inputs in systems theory, a natural cone is needed in the extrapolation space $X_{-1}$. 
Following \cite{BatkaiJacobVoigtWintermayr2018, Wintermayr2019} we define the cone $X_{-1,+}$ in $X_{-1}$ as 
\begin{align}
    \label{eq:cone-X_-1}
    X_{-1,+}
    := 
    \overline{X_+}^{X_{-1}}
    ,
\end{align}
i.e., as the closure of $X_+$ within $X_{-1}$. 
This is indeed a cone and it satisfies $X_{-1,+} \cap X = X_+$ \cite[Proposition~2.3]{BatkaiJacobVoigtWintermayr2018}. 
The extended resolvent operators $(\mu - A_{-1})^{-1}$ are positive from $X_{-1}$ to $X$ for each $\mu$ that is larger than the real part of all spectral values of $A$ \cite[Remark~2.2]{BatkaiJacobVoigtWintermayr2018}. 
Various order properties of the cone $X_{-1,+}$, even in the case that $X$ is an ordered Banach space, were recently proved in \cite[Section~2.2]{AroraGlueckPaunonenSchwenninger2024}.

As mentioned in the introduction, $X_{-1,+}$ is usually not generating in $X_{-1}$ and so $X_{-1}$ is not a vector lattice. 
However, because every cone of an ordered Banach space is generating in its span, it is reasonable to ask if the span of $X_{-1,+}$ is a vector lattice. 
This is indeed often the case, as the following consequence of Theorem~\ref{thm:span-lattice-oc} shows.

\begin{theorem}
    Let $(T(t))_{t\ge 0}$ be a positive $C_0$-semigroup on a Banach lattice $X$ with order continuous norm. 
    Then $\linSpan(X_{-1,+})$ is a Banach lattice with order continuous norm with respect to a norm equivalent to $\norm{\argument}_{\linSpan(X_{-1,+})}$. 
    Moreover, the canonical embedding of $X$ into $\linSpan(X_{-1,+})$ is a lattice homomorphism and $X$ is a lattice ideal in $\linSpan(X_{-1,+})$.
    If $X$ is a KB-space, then so is $\linSpan(X_{-1,+})$.
\end{theorem}

\begin{proof}
    Apply Theorem~\ref{thm:span-lattice-oc} to the space $Z := X_{-1}$ with $J: X \to Z$ being the canonical embedding, $R_n := n (n-A_{-1})^{-1}$ for all sufficiently large integers $n$, and $S := R_{n_0}$ for some fixed $n_0$.
\end{proof}

In some situations one can concretely determine the space $\linSpan(X_{-1,+})$:

\begin{example}
    \label{exa:neumann-laplace}
    Let $\Omega \subseteq \bbR^d$ be a bounded domain with Lipschitz boundary, let $p,q \in (1,\infty)$ be Hölder conjugates, and let $\Delta$ denote the Neumann Laplace operator on $X = L^p(\Omega)$. 
    If $d < 2q$, then $\linSpan(X_{-1,+})$ is precisely the space of finite Borel measures on $\overline{\Omega}$; see \cite[Example~2.13(b)]{AroraGlueckPaunonenSchwenninger2024}. 
    Note that, since positive operators on ordered Banach spaces with generating cone are automatically continuous \cite[Theorem~2.32]{AliprantisTourky2007}, all complete norms that turn $\linSpan(X_{-1,+})$ into an ordered Banach space, are equivalent. 
    Hence, the norm $\norm{\argument}_{\linSpan(X_{-1,+})}$ is equivalent to the total variation norm in this example and consequently, the space $\linSpan(X_{-1,+})$ is not reflexive despite $X$ and $X_{-1}$ both being reflexive.

    Since the assumptions of Corollary~\ref{cor:span-reflexive} are also satisfied in this example (with $Z := X_{-1}$) this also shows that the conclusion of Corollary~\ref{cor:span-reflexive} cannot be improved to give that $\linSpan(Z_+)$ is reflexive.
\end{example}

Let us conclude this section with another example of an extrapolation space, which is simpler but nevertheless illuminating.

\begin{example}
    \label{exa:multiplication-op}
    Let $p \in [1,\infty)$, let $(\Omega,\mu)$ be a measure space, 
    and endow $L^p(\Omega,\mu)$ with the pointwise almost everywhere order. 
    Let $m: \Omega \to [0,\infty)$ be measurable. 
    The operator $A: f\mapsto -mf$ on $L^p(\Omega,\mu)$ with domain 
    \begin{align*}
        \dom{A} := \{f \in L^p(\Omega,\mu): \, mf \in L^p(\Omega,\mu)\}
    \end{align*}
    generates a positive $C_0$-semigroup $(T(t))_{t \ge 0}$ given by 
    \begin{align*}
        T(t)f = e^{-tm}f
    \end{align*}
    for each $f \in L^p(\Omega,\mu)$ and each $t \in [0,\infty)$. 
    It is easy to check that the extrapolation space $\big(L^p(\Omega,\mu)\big)_{-1}$ equals $L^p(\Omega,\nu)$ with $\dxShort \nu = \frac{1}{(m+1)^p}\dxShort \mu$, and that the cone~\eqref{eq:cone-X_-1} in the extrapolation space consists simply of the functions in $L^p(\Omega, \nu)$ that are $\ge 0$ $\nu$-almost everywhere.
\end{example}

In contrast to Example~\ref{exa:neumann-laplace}, the cone $X_{-1,+}$ in the extrapolation space $X_{-1} := \big(L^p(\Omega,\mu)\big)_{-1}$ in Example~\ref{exa:multiplication-op} is generating, so $\linSpan(X_{-1,+}) = X_{-1}$. 

\section{Outlook}

Assume that we are in the situation of Theorem~\ref{thm:span-lattice-oc}, but that the Banach lattice $X$ is not required to have order continuous norm. 
We do not know whether $\linSpan(Z_+)$ still needs to be a Banach lattice in this case (the arguments in the proof of Theorem~\ref{thm:span-lattice-oc} show that it suffices to prove that $\linSpan(Z_+)$ is a vector lattice). 
This is the first of the following two open problems.

\begin{open_problems}
    \leavevmode
        (a)
        Let $X$ be a Banach lattice and let $Z$ be an ordered Banach space.
        Assume that there exists a positive bijection $S \in \calL(Z,X)$, a positive operator $J \in \calL(X,Z)$ and a sequence of positive operators $(R_n) \subseteq \calL(Z,X)$ such that $(JR_n) \subseteq \calL(Z)$ converges to $\id_Z$ with respect to the strong operator topology. 
        (If it is helpful, one might also want to assume that, in addition, $R_n J \to \id_X$ strongly.)

        Does it follow that $\linSpan(Z_+)$ is a vector lattice (and thus a Banach lattice with respect to a norm equivalent to $\norm{\argument}_{\linSpan(Z_+)})$?

        (b)
        If the answer to problem~(a) is negative, does it become positive if the mapping $J$ is compact?
    
\end{open_problems}

As explained before the problem statement, Theorem~\ref{thm:span-lattice-oc} says that the answer to problem~(a) is positive if $X$ has order continuous norm. 
Two concrete examples in which $X$ does not have order continuous norm are given in \cite[Examples~5.1 and~5.3]{BatkaiJacobVoigtWintermayr2018}. 
In both examples, $X$ is a space of continuous functions on an interval (so it does not have order continuous norm) and $Z := X_{-1}$ is the extrapolation space of a positive $C_0$-semigroup on $X$. 
However, in both examples the positive cone is computed explicitly in \cite[formulas~(5.4) and~(5.6)]{BatkaiJacobVoigtWintermayr2018} and those results show that the span of the cone is indeed a lattice. 
In both examples the canonical embedding from $X$ into $X_{-1}$ is compact (since the generator of the semigroup has compact resolvent in both cases). 
This might be an indication that the answer to problem~(b) above is positive.

\section*{Acknowledgements} 
The first author was funded by the Deutsche Forschungsgemeinschaft (DFG, German Research Foundation) -- 523942381.
The second author is indebted to the Department of Applied Mathematics of the University of Twente for a very pleasant stay during which some parts of the work on this article were done.
The authors are grateful to Lassi Paunonen for useful comments on the manuscript. 
The authors are also indebted to Jens Wintermayr for asking whether the reflexivity assumption in Corollary~\ref{cor:span-reflexive} can be weakened, a question that inspired Theorem~\ref{thm:span-lattice-oc}.

\bibliographystyle{plainurl}
\bibliography{literature}

\end{document}